\documentclass{amsart}
\usepackage{amsmath,color}
\usepackage{amssymb}
\usepackage{amscd}
\usepackage{hyperref}
\usepackage{longtable}
\newtheorem{theorem}{Theorem}[section]
\newtheorem{lemma}[theorem]{Lemma}

\theoremstyle{definition}
\newtheorem{definition}[theorem]{Definition}
\newtheorem{definitions}[theorem]{Definitions}
\newtheorem{example}[theorem]{Example}

\newtheorem{definitions and remarks}[theorem]{Definitions and Remarks}

\newtheorem{notation}[theorem]{Notation}

\theoremstyle{remark}
\newtheorem{remark}[theorem]{Remark}
\newtheorem{remarks}[theorem]{Remarks}

\numberwithin{equation}{section}

\newcommand{\inv}{\mathrm{inv}}

\newcommand{\cosupp}{\mathrm{cosupp}\,}

\newcommand{\ord}{\mathrm{ord}}

\newcommand{\s}{{\sigma}}

\newcommand{\IN}{{\mathbb N}}

\newcommand{\IA}{{\mathbb A}}

\newcommand{\cE}{{\mathcal E}}

\newcommand{\cG}{{\mathcal G}}

\newcommand{\cI}{{\mathcal I}}
\newcommand{\cJ}{{\mathcal J}}
\newcommand{\cM}{{\mathcal M}}

\newcommand{\cO}{{\mathcal O}}

\newcommand{\cR}{{\mathcal R}}
\newcommand{\cS}{{\mathcal S}}

\newcommand{\tH}{{\widetilde H}}

\newcommand{\ucG}{\underline{\cG}}

\newcommand{\ucI}{\underline{\cI}}
\newcommand{\ucJ}{\underline{\cJ}}

\begin{document}
\title[Desingularization avoiding simple normal crossings]{Desingularization by blowings-up avoiding simple normal crossings}

\author[E. Bierstone]{Edward Bierstone}
\address{The Fields Institute, 222 College Street, Toronto, ON, Canada
M5T 3J1, and University of Toronto, Department of Mathematics, 40 St. George Street,
Toronto, Ontario, Canada M5S 2E4}
\email{bierston@fields.utoronto.ca}
\thanks{Research supported in part by NSERC grants OGP0009070,
MRS342058, USRA191085, CGS145594585 and OGP0008949.}

\author[S. Da Silva]{Sergio Da Silva}
\address{University of Toronto, Department of Mathematics, 40 St. George Street,
Toronto, Ontario, Canada M5S 2E4}
\email{sergio.dasilva@utoronto.ca}

\author[P. Milman]{Pierre D. Milman}
\address{University of Toronto, Department of Mathematics, 40 St. George Street,
Toronto, Ontario, Canada M5S 2E4}
\email{milman@math.toronto.edu}

\author[F. Vera Pacheco]{Franklin Vera Pacheco}
\address{The Fields Institute, 222 College Street, Toronto, ON, Canada
M5T 3J1}
\email{franklin.vp@gmail.com}

\subjclass{Primary 14E15, 14J17, 32S45; Secondary 14B05, 32S05, 32S10}

\keywords{resolution of singularities, simple normal crossings, desingularization invariant}

\begin{abstract}
It is shown that, for any reduced algebraic variety in characteristic zero,
one can resolve all but simple normal crossings (snc) singularities by a finite sequence of
blowings-up with smooth centres which, at every step, avoids points where the 
transformed variety together with the exceptional divisor has only 
snc singularities. The proof follows the philosophy of \cite{BMmin}
that the desingularization invariant can be used together with natural geometric
information to compute local normal forms of singularities.
\end{abstract}
\date{\today}
\maketitle
\setcounter{tocdepth}{1}
\tableofcontents


\section{Introduction}\label{sec:intro}
The subject of this article is partial resolution of singularities. 
The main result asserts that, for any (reduced) algebraic variety $X$ in characteristic zero,
we can resolve all but simple normal crossings singularities by a finite sequence of
blowings-up with smooth centres which, at every step, avoids points where the 
corresponding transform of $X$ together with the exceptional divisor has only 
simple normal crossings singularities. For background and motivation of the problem,
see \cite{BMmin}, \cite{BV} and \cite{Kolog}. Our proof follows the philosophy of \cite{BMmin} that the desingularization invariant of \cite{BMinv}, \cite{BMfunct}  
can be used together with natural geometric
information to compute local normal forms of singularities.

Every algebraic variety can be embedded locally in an affine space. An algebraic
variety $X$ has \emph{simple normal crossings} (\emph{snc}) at a point $a$ if,
locally at $a$, there is an embedding $X \hookrightarrow Z$ (where $Z$ is smooth)
and a regular coordinate system $(x_1,\ldots,x_n)$ for $Z$ at $a$ in which each
irreducible component of $X$ is a coordinate hypersurface $(x_i = 0)$, for some 
$i$. Thus snc singularities are singularities of \emph{hypersurfaces}. ($X$ is a \emph{hypersurface} if, locally, $X$ can be defined by a principal ideal sheaf on
a smooth variety).

Our main problem can be reduced to the case that $X$ is a closed hypersurface in 
a smooth variety $Z$, essentially
because the desingularization algorithm of \cite{BMinv}, \cite{BMfunct} blows up non-hypersurface points first. (A detailed argument can be found in \cite{BMmin}, \cite{BV}.) 
We therefore state and prove the main result here only in the case
of an embedded hypersurface (Theorem \ref{thm:main}).

\begin{definitions}\label{def:snc}
Let $X \hookrightarrow Z$ denote an embedded hypersurface ($Z$ smooth) and
let $E$ denote a divisor on $Z$. We say that $(X,E)$ has (or is) \emph{simple
normal crossings} (\emph{snc}) at a point $a \in Z$ if $Z$ admits a regular system of
coordinates $(x_1,\ldots, x_p, y_1,\ldots, y_q)$ (where $p \geq 0$) at $a$, in which
the irreducible components of $X$ at $a$ are $(x_i = 0)$, $i=1,\dots,p$, and the support 
of each component of $E$ at $a$ is $(y_j=0)$, for some $j$. We also say that $X$
(respectively, $E$) is \emph{snc} at $a$ if $(X,\emptyset)$ (respectively, 
$(\emptyset, E)$)
is snc at $a$. The pair $(X,E)$ is \emph{snc} if it is snc at every point.
\end{definitions}

Consider $X \hookrightarrow Z$ and $E$ as in Definitions \ref{def:snc}. We consider
the support of $E$ as a divisor $\sum_{j=1}^r H_j$ (where each component $H_j$ is
a closed hypersurface in $Z$, not necessarily irreducible), or, equivalently, as the
collection $\{H_1,\ldots, H_r\}$. The results in this article involve only the support of $E$, 
so we will write $E = \sum_{j=1}^r H_j$ or $E = \{H_1,\ldots, H_r\}$, by abuse of notation.

\begin{notation}\label{not:seq}Given a sequence of blowings-up
\begin{equation}\label{eq:blup}
Z = Z_0 \stackrel{\s_1}{\longleftarrow} Z_1 \longleftarrow \cdots
\stackrel{\s_{t}}{\longleftarrow} Z_{t}\,,
\end{equation}
where each $\s_{j+1}$ has smooth centre $C_j \subset Z_j$, we write
$X_0 = X$, $E_0 = E$, and, for each $j = 0,1,\ldots$\,, we set
\begin{enumerate}
\item[]
\begin{enumerate}
\item[$X_{j+1} :=$]  strict transform of $X_j$,
\smallskip
\item[$E_{j+1} :=$]  collection of strict transforms of the components of $E_j$
together with the exceptional divisor $\s_{j+1}^{-1}(C_j)$ of $\s_{j+1}$. 
\end{enumerate}
\end{enumerate}
This notation will be used throughout the article.
\end{notation}

\begin{definition}\label{def:admiss}
A blowing-up $\s: Z' \to Z$ is \emph{admissible} (for $(X,E)$) 
if the centre $C$ of $\s$ is smooth, $C \subset X$
and $C$ is \emph{snc with respect to} $E$ (where the latter means that,
locally in regular coordinates, $C$ is a coordinate subspace and the support of each
component of $E$ is a coordinate hypersurface. Note that $C$ may lie in certain
components of $E$.)
The sequence of blowings-up \eqref{eq:blup} is called \emph{admissible}
if each $\s_{j+1}$ is admissible for $(X_j,E_j)$. 
\end{definition}

If $E_j$ is snc and $\s_{j+1}$ is admissible, then $E_{j+1}$ is snc. 

\begin{theorem}\label{thm:main}
Let $X \hookrightarrow Z$ denote an embedded hypersurface in characteristic zero and
let $E$ denote a snc divisor on $Z$. Then there is a sequence of admissible
blowings-up \eqref{eq:blup} such that
\begin{enumerate}
\item $(X_t, E_t)$ has only snc singularities,
\item each $\s_{j+1}$ is an isomorphism over the locus of snc points of $(X_j,E_j)$.
\end{enumerate}
Moreover, the association of the desingularization sequence \eqref{eq:blup} to 
$(X,E)$ is functorial with respect to 
smooth morphisms that preserve the number of irreducible components of $X$
at every point.
\end{theorem}

Condition (2) of the theorem implies that each centre $C_j \subset X_j$.

See \cite[Section 12]{BMinv} for an earlier approach to Theorem \ref{thm:main}.
Weaker versions of Theorem \ref{thm:main}, where (2) is replaced by the
condition that the morphism $\s_t \circ \cdots \circ \s_1$ is an isomorphism over
the snc locus of $(X,E)$, can also be found in \cite{Sz}, \cite{Kolog} and
\cite[Thm.\,3.4]{BMmin} (the latter is functorial as in Theorem \ref{thm:main}).
Our proof of Theorem \ref{thm:main} uses the desingularization invariant
$\inv = \inv_{(X,E)}$, in the spirit of \cite{BMmin}. Theorem \ref{thm:main} is an
important special case of the main theorem of \cite{BV} and is used in the proof of the
latter (see \cite[Rmk.\,1.3(4)]{BV}).

We refer to the Appendix of \cite{BMmin}, \emph{Crash course on the desingularization
invariant}, for the definition of $\inv$ and how to compute it (see also \cite{BMinv},
\cite{BMfunct}), though some of the ideas will be recalled in Section 2 below. 
The invariant $\inv$
is defined iteratively over a sequence of blowings-up \eqref{eq:blup} which are 
$\inv$-\emph{admissible} (where the latter means that the successive blowings-up
are admissible and $\inv$ is constant on each component of a centre).

A key ingredient in our proof of Theorem \ref{thm:main} is a characterization of
snc (in any \emph{year} of the desingularization history \eqref{eq:blup}) in terms of
certain special values $\inv_{p,s}$ of $\inv$ (Definition \ref{def:specialinv}). 
The blowings-up in 
Theorem \ref{thm:main} are obtained by following the standard desingularization
algorithm \cite{BMinv}, \cite{BMfunct} until the maximum value of $\inv$ becomes
one of the special values $\inv_{p,s}$. When $\inv$ attains a value $\inv_{p,s}$
at a point $a$, we can write a local normal form for the singularity (Lemma \ref{lem:norm})
and then simplify the normal form by a sequence of invariantly-defined \emph{cleaning}
blowings-up (cf. \cite[Sect.\,2]{BMmin}) until the singularity becomes snc. Cleaning
blowings-up are admissible but not, in general, $\inv$-admissible. Nevertheless,
a modified version of $\inv$ can be defined over the particular cleaning
sequences needed in this article (Section 4). This is enough to be able to repeat
the preceding process on the complement of the snc locus until finally we resolve
all but snc singularities. 

Section 2 includes an example that illustrates the difference between
the algorithm of Theorem \ref{thm:main} and the weaker version of \cite[Thm.\,3.4]{BMmin}
or the standard desingularization algorithm of \cite{BMinv}.


\section{The desingularization invariant and examples}\label{sec:inv}
The {\it Crash course on the desingularization invariant} \cite[Section 5]{BMmin} is a
prerequisite for our proof of Theorem \ref{thm:main}, and should be consulted for the
notions of \emph{maximal contact}, \emph{coefficient ideal}, \emph{companion ideal},
etc. We will recall a few of the ideas involved in order to fix notation.

Resolution of singularities in characteristic zero can be realized essentially 
by choosing, as each successive centre of blowing up, the locus of maximal values of the desingularization
invariant $\inv = \inv_{(X,E)}$ \cite{BMinv}, \cite{BMfunct}. The invariant $\inv$ is
defined iteratively over a sequence of $\inv$-admissible blowings-up \eqref{eq:blup}.
In particular, if $a \in Z_j$, then $\inv(a)$ depends on the previous blowings-up. 

Let $a \in Z_j$. Then $\inv(a)$ has the form
\begin{equation}\label{eq:inv}
\inv(a) = (\nu_1(a), s_1(a), \ldots, \nu_q(a), s_q(a), \nu_{q+1}(a))\,,
\end{equation}
where $\nu_k(a)$ is a positive rational number if $k\leq q$, each $s_k(a)$
is a nonnegative integer, and $\nu_{q+1}(a)$ is either $0$ (the order of the unit
ideal) or $\infty$ (the order of the zero ideal). The first entry
$\nu_1(a)$ is the order $\ord_a X_j$ of $X_j$ at $a$. If $a \in Z_0$ (i.e, in ``year
zero''), $s_1(a)$ is the number of components of $E$ at $a$.
The successive pairs $(\nu_k(a),s_k(a))$ are defined inductively over
\emph{maximal contact} subvarieties of increasing codimension. 
$\inv(a) = (0)$
if and only if $a \in Z_j\setminus X_j$. 

We order finite sequences of the form (\ref{eq:inv}) lexicographically.
Then $\inv(\cdot)$ is upper-semicontinuous on each $Z_j$, and 
\emph{infinitesimally upper-semicontinuous}, i.e., if $a \in Z_j$,
then $\inv(\cdot) \leq \inv(a)$ on $\s_{j+1}^{-1}(a)$. 

We also introduce truncations of $\inv$. Let $\inv_{k+1}(a)$ denote
the truncation of $\inv(a)$ after $s_{k+1}(a)$ (i.e., after the $(k+1)$st 
pair), and let $\inv_{k+1/2}(a)$ denote the truncation of $\inv(a)$ after
$\nu_{k+1}(a)$. 

Given $a \in Z_j$, let $a_i$ denote the image of $a$ in $Z_i$, $i\leq j$.
(We will speak of \emph{year} $i$ in the history of blowings-up). 
The \emph{year of birth} of 
$\inv_{k+1/2}(a)$ (or $\inv_{k+1}(a)$) denotes the smallest $i$ such that
$\inv_{k+1/2}(a) = \inv_{k+1/2}(a_i)$ (respectively, $\inv_{k+1}(a) =
\inv_{k+1}(a_i)$). 

Let $a \in Z_j$. Let $E(a)$ denote the set of components of $E_j$ which
pass through $a$. The entries $s_k(a)$ of $\inv(a)$ are the sizes of certain 
subblocks of $E(a)$: Let $i$ denote the birth-year of 
$\inv_{1/2}(a) = \nu_1(a)$, and let $E^1(a)$ denote the collection of
elements of $E(a)$ that are strict transforms of components of $E_i$
(i.e., strict transforms of elements of $E(a_i)$). Set $s_1(a) :=
\#E^1(a)$. We define $s_{k+1}(a)$, in general, by induction on $k$: Let $i$ denote
the year of birth of $\inv_{k+1/2}(a)$ and let $E^{k+1}(a)$ denote the 
set of elements of $E(a) \setminus \left(E^1(a) \cup \cdots \cup E^k(a)\right)$
that are strict transforms of components of $E_i$. Set $s_{k+1}(a) :=
\#E^{k+1}(a)$. Relative to $\inv_{k+1/2}(a)$, the elements of $E^{k+1}(a)$ are ``old"
components of the exceptional divisor, and 
$\cE^{k+1}(a) := E(a) \setminus \left(E^1(a) \cup \cdots \cup E^{k+1}(a)\right)$ is the
set of ``new'' components.

The successive pairs in $\inv(a)$ are calculated using ``marked ideals'' --- collections
of data that are computed iteratively on ``maximal contact'' subspaces of increasing
codimension. A \emph{marked ideal} $\ucI$ is a quintuple $\ucI = (Z,N,E,\cI,d)$, 
where $Z \supset N$ are smooth varieties,
$E = \sum_{i=1}^s H_i$ is a simple normal crossings divisor on $Z$ which is
tranverse to $N$,
$\cI \subset \cO_N$ is an ideal, and $d \in \IN$. (See \cite[Defns.\,5.5]{BMmin}.)

The \emph{cosupport} of $\ucI$,
$\cosupp \ucI := \{x \in N:\, \ord_x\cI \geq d\}$.
We say that $\ucI$ is of \emph{maximal
order} if $d = \max\{\ord_x \cI: x \in \cosupp \ucI\}$. 

A blowing-up $\s: Z' \to Z$ (with smooth centre $C$) is
\emph{$\ucI$-admissible} (or simply \emph{admissible}) if
$C \subset \cosupp \ucI$, and
$C$, $E$ have only normal crossings.
The \emph{(controlled) transform} of $\ucI$ by an admissible blowing-up
$\s: Z' \to Z$ is the marked ideal $\ucI' = (Z',N',E',\cI',d'=d)$,
where $N'$ is the strict transform of $N$ by $\s$,
$E' = \sum_{i=1}^{s+1} H_i'$ (where $H_i'$ is the strict transform
of $H_i$, for each $i=1,\ldots,s$, and $H'_{s+1}$ is the exceptional divisor
$\s^{-1}(C)$ of $\s$),
$\cI' := \cI_{\s^{-1}(C)}^{-d}\cdot \s^*(\cI)$ (where $\cI_{\s^{-1}(C)}
\subset \cO_{N'}$ denotes the ideal of $\s^{-1}(C)$).
In this definition, note that $\s^*(\cI)$
is divisible by $\cI_{\s^{-1}(C)}^d$ and $E'$ is a normal crossings 
divisor transverse to $N'$, because $\s$ is admissible.

Computation of $\inv$ and of the centre of blowing up begins,
in year zero, with the marked ideal $\ucI^0 = \ucI_X := (Z,Z,E,\cI_X,1)$,
where $\cI_X \subset \cO_Z$ is the ideal of $X$. Consider a point $a \in X_j$, 
in an arbitrary year $j$. We will simplify notation by writing $Z,\,X,\,E$, etc.,
instead of $Z_j,\,X_j,\,E_j$, etc., when the year is understood. At $a$, the
computation begins with $\ucI^0 = (Z,Z,\cE,\cI^0,1)$, restricted to
some neighbourhood $U$ of $a$, where $\cE = \cE(a) = E(a)$ and $\cI^0$ denotes 
the transform of $\cI_{X_0}$ (iterating the definition above). 
Then $\cI^0$ equals the product of $\cI_X = \cI_{X_j}$ 
and a monomial in generators of the
ideals of the elements of $\cE$. 

We pass from $\ucI^0$ to the \emph{companion ideal} $\ucJ^0 = \ucG(\ucI^0)$,
which is obtained from the product decomposition 
$\cI^0 = \cM(\ucI^0)\cdot\cR(\ucI^0)$, where $\cM(\ucI^0)$ is the \emph{monomial
part} of $\cI^0$ (it is a monomial in exceptional divisors) and $\cR(\ucI^0)$ is
the \emph{residual part} of $\cI^0$ (divisible by no exceptional divisor).
See \cite[\S\S5.6,\,5.8]{BMmin}. Here, clearly 
$\cR(\ucI^0) = \cI_X$ and, since the associated multiplicity of $\ucI^0$ 
is $1$, 
$\ucJ^0 = (\cE, \cR(\ucI^0), d) = (\cE, \cI_X, d) := (Z,Z,\cE, \cI_X, d)$, where
$d = \ord_a\cI_X$. We set  $inv_{1/2}(a) = \nu_1(a) = d$.

Since $\ucJ^0$ is of maximal order, we can pass to the \emph{coefficient ideal
plus boundary} $\ucI^1 = (Z, N^1, \cE^1, \cI^1, d^1)$ defined on a 
\emph{maximal contact hypersurface} $N^1$ of $Z$ 
(in some neighbourhood of $a$); see \cite[\S5.9]{BMmin}. Here $\cE^1 = \cE^1(a)
= E(a)\setminus E^1(a)$. The \emph{boundary} is the marked ideal determined by
the \emph{old} components $E^1(a)$ of the exceptional divisor $E(a)$. 
The maximal contact hypersurface
is transformed from the year of birth of $\inv_{1/2}(a)$, so it is transverse
to the set of \emph{new} components $\cE^1(a)$ of $E(a)$. It need not be tranverse
to $E^1(a)$; the boundary is therefore added to the coefficient ideal to ensure
the centre of blowing up will lie inside the old components.

Iterating the process, we will get a coefficient ideal plus boundary
$\ucI^k = (Z,N^k,\cE^k,\cI^k,d^k)$ defined
on a maximal contact subspace $N^k$ of codimension $k$ in $Z$, in a neighbourhood
of $a$. The boundary here is given by $E^k(a)$, and
$\cE^k = \cE^k(a) :=
\cE^{k-1}(a)\setminus E^k(a) = E(a) \setminus E^1(a) \cup \cdots \cup E^k(a)$.

We again write $\cI^k$ as the product $\cM(\ucI^k)\cdot\cR(\ucI^k)$ of its
monomial and residual parts. $\cM(\ucI^k)$ is the monomial part with respect
to $\cE^k$; i.e., the product of the ideals
$\cI_H$, $H \in \cE^k(a)$, each to the power $\ord_{H,a}\cI^k$,
where $\ord_{H,a}$ denotes the order along $H$ at $a$. 
We set $\nu_{k+1}(a) := \ord_a \cR(\ucI^k)/d^k$ and 
$\mu_{H,k+1}(a) := \ord_{H,a}\cI^k/d^k$, $H \in \cE^k(a)$; both are invariants
of the \emph{equivalence class} of $\ucI^k$ and $\dim N^k$
(see \cite[Def.\,5.10]{BMmin}). The iterative construction terminates when
$\nu_{k+1}(a) = 0$ or $\infty$.

In Section \ref{sec:char} (see Lemma \ref{lem:snc}), we will show that, at
an snc point in any year of the resolution history, the invariant
takes a special form:  

\begin{definition}\label{def:specialinv}
Given $p \in \IN$, $d\geq1$ and $s=(s_1,...,s_d)$ with each $s_k \in \IN$, 
set
$$ 
\inv_{p,s}:=(p,s_1,1,s_2,\ldots,1,s_d,1,0,\ldots,1,0,\infty),
$$
where the total number of pairs (before $\infty$) is
$$ 
r:= p+|s|,\quad |s| := \sum_{k=1}^ds_k.
$$
\end{definition}

In particular, if $E = \emptyset$, then, at an snc point in year zero,
$\inv = \inv_{p,0} = (p,0,1,0,\dots,1,0,\infty)$ with $p$ pairs. 
The algorithm of \cite[Thm.\,3.4]{BMmin} depends
on a characterization of points with $\inv = \inv_{p,0}$.

The following example distinguishes the algorithm of 
Theorem \ref{thm:main} (which we will call Algorithm C) from those of
the weaker result \cite[Thm.\,3.4]{BMmin} (Algorithm B) and the
standard desingularization algorithm \cite{BMinv}, \cite{BMfunct} (Algorithm
A). See also Example \ref{ex2}. The table below provides
the computations of marked ideals
needed to find the invariant and the centre $C$ of the blowing-up
at the origins of the charts corresponding to the coordinate substitutions
indicated. The calculations at a given point provide the next centre
of blowing up over a neighbourhood of that point; globally, the maximum
locus of the invariant will be blown up first.

In each subtable, the passage from $\ucJ^k$ to $\ucI^{k+1}$ is given
by taking the coefficient ideal plus boundary, on the maximal contact
subspace of codimension $k+1$.

\begin{example}\label{ex1}
Consider the hypersurface $X \hookrightarrow  \IA^3$ given by 
$(z^3+xy = 0)$, together with the empty divisor $E = \emptyset$. The 
following table computes blowings-up given by the standard resolution
algorithm A.
\smallskip

\renewcommand{\arraystretch}{1.5}
\begin{longtable}{c | c | c | c | c}
\hline
 codim & marked ideal & companion ideal  & maximal & boundary\\[-.25cm]
$i$ & $\ucI^i$ & $\ucJ^i = \ucG(\ucI^i)$ & contact & $E^i$ \\\hline
\multicolumn{5}{ c }{}\\[-.5cm]
\multicolumn{5}{ l }{Year zero.}\\[.1cm]\hline
0 & $(z^3+xy,1)$ & $(z^3+xy,2)$ & $(x=0)$&\\\hline
1 & $((z^3,y^2),2)$ & $((z^3,y^2),2)$        & $(x=y= 0)$&\\\hline
2 & $(z^3,2)$ & $(z^3,3)$ & $(x=y=z=0)$&\\\hline
3 & $(0)$ &&&\\\hline
\multicolumn{5}{ l }{$\inv(0) = (2,0,1,0,3/2,0,\infty)$,\,  $C_0 = \{0\}$}\\
\multicolumn{5}{ c }{}\\[-.4cm]
\multicolumn{5}{ l }{Year one.\,  Coordinate chart $(xz,yz,z)$}\\[.1cm]\hline
0 & $(z(z+xy),1)$ & $(z+xy,1)$ & $(z=0)$& $(z=0)$\\\hline
1 & $(xy,1)$ & $(xy,2)$        & $(x=z=0)$&\\\hline
2 & $(y,1)$ & $(y,1)$ & $(x=y=z=0)$&\\\hline
3 & $(0)$ &&&\\\hline
\multicolumn{5}{ l }{$\inv(0) = (1,1,2,0,1,0,\infty)$,\,  $C_1 = \{0\}$}\\
\multicolumn{5}{ c }{}\\[-.4cm]
\multicolumn{5}{ l }{Year two.\,  Coordinate chart $(x,xy,xz)$}\\[.1cm]\hline
0 & $(xz(z+xy),1)$ & $(z+xy,1)$ & $(z=0)$& $(z=0)$\\\hline
1 & $(xy,1)$ & $(y,1)$        & $(y=z=0)$& $(x=0)$\\\hline
2 & $(x,1)$ & $(x,1)$ & $(x=y=z=0)$&\\\hline
3 & $(0)$ &&&\\\hline
\multicolumn{5}{ l }{$\inv(0) = (1,1,1,1,1,0,\infty)$,\,  $C_2 = \{0\}$}\\
\multicolumn{5}{ c }{}\\[-.4cm]
\multicolumn{5}{ l }{Year three.\,  Coordinate chart $(x,xy,xz)$}\\[.1cm]\hline
0 & $(x^2z(z+xy),1)$ & $(z+xy,1)$ & $(z=0)$& $(z=0)$\\\hline
1 & $(xy,1)$ & $(y,1)$        & $(y=z=0)$&\\\hline
2 & $(0)$ &&&\\\hline
\multicolumn{5}{ l }{$\inv(0) = (1,1,1,0,\infty)$,\,  $C_3 = \{y=z=0\}$}\\
\end{longtable}

Note that none of the invariants computed are special values of the
form $\inv_{p,0}$. Therefore, Algorithm B coincides with A 
at each of the steps shown. In year three, the centre $C_3$
includes snc points (though not snc points that were present in year
zero), so that both Algorithms A and B blow up snc points in year three.

On the other hand, in year two, $\inv(0)$ is a special value $\inv_{p,s}$,
 where $p=1$ and $s=(1,1)$. At this point, Algorithm C continues
with a cleaning blowing-up, with centre $(x=z=0)$, after which we have
simple normal crossings over the year-two chart.
\end{example}


\section{Characterization of simple normal crossings}\label{sec:char}
Consider $X \hookrightarrow Z$ and $E$ as in Definitions \ref{def:snc}, and
the sequence of $\inv$-admissible blowings-up given by the standard desingularization
algorithm. Lemma \ref{lem:snc} below asserts that $\inv$ takes a special form
$\inv_{p,s}$ at an snc point in any year of the resolution history. The converse is not
true. For example, if $X=(x_1^n+...+x_n^n=0)$ and $E = \emptyset$, then 
$\inv(0)=(n,0,1,0,1,....,0,\infty)=\inv_{n,0}$ (in year zero). However, if we make the
additional assumption that $X$ (or its strict transform in a given year) has $p$ irreducible
components at a point $a$ with $\inv(a) = \inv_{p,s}$, then we can write local normal forms
for the components of $X$ and $E$ at $a$ (Lemma \ref{lem:norm}) and we can
characterize snc using $\inv$ and the additional invariants 
$\mu_{H,k}(a)$ \cite[Def.\,5.10]{BMmin}
(Theorem \ref{thm:char} below). We begin by stating all three results. The proofs of Lemmas \ref{lem:snc} and \ref{lem:norm} follow parallel arguments, so we give them together.

\begin{lemma}\label{lem:snc}
Suppose $(X,E) = (X_m,E_m)$ is snc at a point $a$, in some year $m$ of the
resolution history \eqref{eq:blup}.
Then $\inv(a)$ is of the form $\inv_{p,s}$ where $r = p + |s| \leq n$ (see Definition
\ref{def:specialinv}). Moreover, the invariants $\mu_{H,i+1}(a)=0$, for all 
$i \geq 1$ and $H \in E(a)$.
\end{lemma}

\begin{definition}\label{def:sig}
Let $\Sigma_p = \Sigma_p(X)$ denote the set of points lying in $p$ irreducible 
components of $X$.
\end{definition}

\begin{lemma}\label{lem:norm}
Let $a\in X = X_m$. Assume that $a\in\Sigma_{p}(X)$ and $\inv(a)=\inv_{p,s}$,
$s=(s_1,\ldots,s_d)$. Let $f_k$, $k=1,\ldots,p$, denote generators of the ideals of the components of $X$ at $a$, and let $u_{i}^{j}$, $j=1,\ldots,s_i$, denote generators of the ideals of the elements of $E^i(a)$, $i=1,...,d$. 

Set $r:=p+|s|$. Then there is a bijection $\{1,\ldots,r\}\rightarrow \{f_k,u_{i}^{j}\}$, which
we denote  $l\mapsto g_l$, and a regular system of coordinates $(x_1,\ldots,x_n)$ at 
$a$ ($n\geq r$), such that  
\begin{equation}\label{eq:localform}
\begin{aligned}
 g_1&=f_1=x_1\\
g_l&=\xi_l+x_l\prod_{i=1}^{l-1}m_{i+1},\quad l=2,\ldots,r,
\end{aligned}
\end{equation}
where each $\xi_l$ is in the ideal generated by $(x_1,\ldots, x_{l-1})$ and each $m_{i+1}$ is a monomial in generators of the ideals of the elements $H$ of 
$\mathcal{E}^i(a)= E(a)\setminus E^1(a)\cup ...\cup E^i(a)$, each raised to the power
$\mu_{H,i+1}(a)$. 
\end{lemma}

\begin{theorem}[Characterization of snc]\label{thm:char}
Let $a \in X = X_m$. Then $(X,E)$ is snc at $a$ if and only if
\begin{enumerate}
\item $a\in\Sigma_{p}(X)$, for some $p\geq 1$,
\item $\inv(a)=\inv_{p,s}$, for some $s=(s_1,\ldots,s_d)$,
\item $\mu_{H,i+1}(a) = 0$, for all $i\geq 1$ and all $H \in \cE^i(a)$. 
\end{enumerate}
\end{theorem}

\begin{proof}
Immediate from Lemmas \ref{lem:snc} and \ref{lem:norm}.
\end{proof}

We will need the following simple lemma (see \cite[\S5.4]{BMmin}).

\begin{lemma}\label{lem:equivproductsum}
Let $a \in N \subset Z$, where $N$ is a smooth subvariety of a neighbourhood of $a$, and let $f_1,\ldots,f_p$ denote regular functions of order one on $N$. Then the marked ideals $\underline{\mathcal{I}}_1=(Z,N,\emptyset,(f_1\dotsm f_p), p)$ and $\underline{\mathcal{I}}_2=(Z,N,\emptyset,(f_1,\ldots,f_p),1)$ are equivalent. (In particular,
any sequence of blowings-up which is admissible for one is also admissible for the other.)
\end{lemma}

\begin{proof}[Proof of Lemmas \ref{lem:snc} and \ref{lem:norm}]
Consider $(X,E) = (X_m,E_m)$ in some year $m$ of the resolution history \eqref{eq:blup},
and consider $a \in \Sigma_p(X)$. Let $f_k$, $k=1,\ldots,p$, denote generators of the ideals of the components of $X$ at $a$. The computation of $\inv(a)$ begins with the
marked ideal $\ucI^0 = (Z,Z,E(a),(f_1\cdots f_p),1)$. Clearly, $\inv_{1/2}(a) = p$
if and only if $\ord_a f_k = 1$ for all $k$. Assume the latter. Then the companion
ideal $\ucJ^0 = (E(a),(f_1\cdots f_p),p)$. Any $(f_k = 0)$, say $(f_1=0)$, defines a 
hypersurface of maximal contact $N^1$ for $(\cE(a),(f_1\cdots f_p),p)$ or, equivalently, 
$(E(a),(f_1,\ldots, f_p),1)$. (This is a consequence of Lemma \ref{lem:equivproductsum} and the fact that,
if $m'$ is the year of birth of $\inv_{1/2}(a)$, then
$X_{m'}$ has $p$ components of order 1 at $a_{m'}$.) Let $x_1=g_1:=f_1$.

Let $u_1^j$, $j=1,\ldots, s_1$, denote generators of the ideals of the elements
of $E^1(a)$. Then the coefficient ideal plus boundary is 
\begin{equation}\label{eq:coeff}
\ucI^1 = (Z,N^1,\cE^1(a),(f_2,\ldots,f_p,u_1^1,\ldots,u_1^{s_1})|_{N^1},1),
\end{equation}
where $\cE^1(a) = E(a)\setminus E^1(a)$. (Note, in particular, that the associated multiplicity is 1.)

We factor $\ucI^1$ as the product $\cM(\ucI^1)\cdot\cR(\ucI^1)$ of its monomial
and residual parts; in particular, $\cM(\ucI^1)$ is generated by a monomial $m_2$
in the components of $\cE^1(a)$.

First suppose that $(X,E)$ is snc at $a$.
Then the generators of $\cI^1$ in \eqref{eq:coeff} are part
of a regular coordinate system. It follows that $\cM(\ucI^1) = 1$ (since none of
these generators define elements of $\cE^1(a)$); i.e., all 
$\mu_{H,2}(a) = 0$. Since $\ucI^1$ has maximal order, $\inv_{3/2}(a) = (p,s_1,1)$, and
the next companion ideal $\ucJ^1 = \ucI^1$. 

We can then continue as before,
choosing the $f_k$ and the $u_{i}^{j}$ successively as hypersurfaces of maximal contact
to pass to the coefficient ideal plus boundary 
$\underline{\mathcal{I}}^\ell$, $\ell = 2,\ldots$\,. 
At each step, $\cM(\ucI^\ell) = 1$ (in particular, $\mu_{H,\ell +1}(a)=0$ for every $H$),
and $\underline{\mathcal{I}}^\ell$ is of maximal order, $=1$. 
Therefore, $\nu_{\ell +1}=1$ and 
$\underline{\mathcal{I}}^\ell$ equals the following companion ideal 
$\underline{\mathcal{J}}^\ell$.  
Once all $f_k$ and $u_{i}^{j}$ have been used as hypersurfaces of maximal contact, we 
get coefficient ideal $= 0$. Therefore, $\inv(a)$ has last entry $=\infty$ and $r$ pairs
before $\infty$. This completes the proof of Lemma \ref{lem:snc}.
\smallskip

Now consider the converse direction, for Lemma \ref{lem:norm}. If $\inv_{3/2}(a) = (p,s_1,1)$, then there exists $g_2\in\{f_2,\ldots,f_p,u_1^1,\ldots,u_1^{s_1}\}$ such that $x_2 := m_2^{-1}\cdot g_2|_{N^1}\in \mathcal{R}(\underline{\mathcal{I}}^1)$ has order 1 at $a$, and the next
companion ideal $\ucJ^1 = (Z,N^1,\cE^1(a),\cR(\ucI^1),1)$. We can take $N^2 :=
(x_2=0) \subset N^1$ as the next maximal contact subspace, and write
$g_2 = \xi_2 + x_2m_2$, where $\xi_2 \in (x_1)$. Then the coefficient ideal plus
boundary is
\begin{equation*}
\ucI^2 = \left(Z, N^2, \cE^2(a)=\cE^1(a)\setminus E^2(a), 
\left(\cR(\ucI^1) + (u_2^1,\ldots,u_2^{s_2})\right)|_{N^2}, 1\right).
\end{equation*}

We can again repeat the argument. At each step $\ell \geq 2$, 
the ideal $\mathcal{I}^\ell$ factors as $\mathcal{M}(\underline{\mathcal{I}}^\ell)\cdot\mathcal{R}(\underline{\mathcal{I}}^\ell)$ where $\mathcal{M}(\underline{\mathcal{I}}^\ell)$ is generated by a monomial $m_{\ell+1}$ in the components of $\cE^\ell (a) = E(a)\setminus E^1(a)\cup...\cup E^\ell (a)$. If the truncated invariant
$\inv_{\ell + 1/2}(a) = (\inv_{p,s})_{\ell + 1/2}$ (in particular, $\nu_{\ell +1}(a)=1$), then 
$\mathcal{R}(\underline{\mathcal{I}}^\ell)$ has maximal order 1.  Therefore, the next companion ideal $\underline{\mathcal{J}}^\ell = (Z,N^\ell,\cE^\ell(a),
\mathcal{R}(\underline{\mathcal{I}}^\ell), 1)$, where
$N^\ell = (x_1=\ldots=x_{\ell}=0)$. Therefore, there is $g_{\ell +1}\in\{f_2,\ldots,f_p,u_1^1,u_1^2,\ldots,u_{\ell}^{s_\ell}\}$ such that 
$x_{\ell +1} := \prod_{i=1}^{\ell}m_{i+1}^{-1}\cdot g_{\ell +1}|_{N^\ell}\in 
\mathcal{R}(\underline{\mathcal{I}}^\ell)$ and has order 1. We can take $(x_{\ell +1} = 0)$ 
as the next hypersurface of maximal contact $N^{\ell +1} \subset N^{\ell}$ 
and write $g_{\ell +1}=\xi_{\ell +1}+x_{\ell +1}\prod_{i=1}^{\ell}m_{i+1}$, for some 
$\xi_{\ell +1}\in(x_1,\ldots, x_{\ell})$. 

If $\inv(a) = inv_{p,s}$, then the process ends after $r=p+|s|$ steps (i.e., $r$ successive choices of maximal contact) with $\underline{\mathcal{I}}^{r}=0$. 
\end{proof}

\begin{remark}\label{rem:mult1} 
In the proof above, we have noted that, if $a\in\Sigma_p$ and the truncated invariant
$\inv_{k+1/2}(a)=(\inv_{p,s})_{k+1/2}$, where $0\leq k<r=p+|s|$, then, for every 
$\ell \leq k+1$, the coefficient ideal plus boundary $\underline{\mathcal{I}}^\ell$ (or an equivalent marked ideal) has associated multiplicity $=1$.
\end{remark}


\section{Cleaning}\label{sec:clean}
According to Theorem \ref{thm:char}, if $a\in\Sigma_{p}$ and $\inv(a)=\inv_{p,s}$, 
then $(X,E)$ is snc at $a$ if and only if the invariants $\mu_{H,k+1}(a)=0$, for every 
$k\geq1$. In this section we study the \emph{cleaning} blowings-up necessary to get the latter condition. 

The centres of cleaning blowings-up are not necessarily $\inv$-admissible. In the general cleaning algorithm of \cite[Sect.\,2]{BMmin}, therefore, $\inv$ is not defined
in a natural way over a cleaning sequence, so that, after cleaning, we assume
we are in year zero for the definition of $\inv$. Over the particular cleaning
sequences needed here, however, we can define a modified $\inv$ which
remains semicontinuous and infinitesimally semicontinuous, and show that maximal
contact subspaces exist in every codimension involved; this is a
consequence of Lemma \ref{lem:norm} and Remark \ref{rem:mult1} (see Remarks
\ref{rem:clean}).

Consider a point $a$ in the locus $S := (\inv_k = (\inv_{p,s})_k)$, where  $k\geq 1$
(in any year $m$ of the resolution history \eqref{eq:blup}). In some neighbourhood of $a$,
$S$ is the cosupport of a marked ideal (a coefficient ideal 
plus boundary) $\underline{\mathcal{I}}^k = (\mathcal{I}^k,d^k)
=(Z,N^k,\mathcal{E}^k(a),\mathcal{I}^k,d^k)$, where $N^k$ is a maximal contact subspace
of codimension $k$ and $d^k=1$ (see Remark \ref{rem:mult1}). Recall that 
$\mathcal{E}^k(a)=E(a)\setminus E^1(a)\cup ...\cup E^{k}(a)$. $\mathcal{E}^k(a)$
is the block of \emph{new} components of the exceptional divisor (necessarily transverse
to $N^k$) and the \emph{old} components in the block $E^{k}(a)$ define the boundary.

The ideal $\mathcal{I}^k = \mathcal{M}(\underline{\mathcal{I}}^k )\cdot\mathcal{R}(\underline{\mathcal{I}}^k )$ (the product of its monomial and residual parts). The
monomial part $\mathcal{M}(\underline{\mathcal{I}}^k )$ is the product of the ideals 
$\mathcal{I}_H|_{N^k}$ (where $H \in \cE^k(a)$), each to the power $\mu_{H,k+1}(a)$ 
(since $d^k=1$).

Let $\underline{\mathcal{M}}(\underline{\mathcal{I}}^k)$ denote the monomial
marked ideal $(\mathcal{M}(\underline{\mathcal{I}}^k), d^k)
= (\mathcal{M}(\underline{\mathcal{I}}^k), 1)$. Then $\cosupp \underline{\mathcal{M}}(\underline{\mathcal{I}}^k) \subset \cosupp \ucI^k$ and any admissible sequence of
blowings up of $\underline{\mathcal{M}}(\underline{\mathcal{I}}^k)$ is admissible
for $\ucI^k$. Cleaning is provided by desingularization of the monomial marked ideal
$\underline{\mathcal{M}}(\underline{\mathcal{I}}^k)$ \cite[Sect.\,5, Step II, Case A]
{BMfunct}, \cite[Sect.\,2]{BMmin};
see Definition \ref{def:clean} following.

\begin{remark}\label{rem:sigma}
First note that, if $X$ has constant order on an irreducible subvariety $C$, then,
for any $p$, either $C \subset \Sigma_p(X)$ or $C \cap \Sigma_p(X) = \emptyset$
(by semicontinuity of order).
\end{remark}

\begin{definition}\label{def:clean}
\emph{Cleaning} of the locus $S = (\inv_k = (\inv_{p,s})_k)$ means the sequence
of blowings-up obtained from desingularization of the monomial marked ideal
$\underline{\mathcal{M}}(\underline{\mathcal{I}}^k)$ (in a neighbourhood of
any point of $S$) by using, at each step, 
only the components of the centres of blowing up which lie in $\Sigma_p$.
\end{definition}

The centres of the cleaning blowings-up are invariantly
defined closed subspaces of $(\inv_k \geq (\inv_{p,s})_k)$.

\begin{remarks}\label{rem:clean}
The blowings-up $\s$ involved in  desingularization of $\underline{\mathcal{M}}(\underline{\mathcal{I}}^k)$ are admissible: Let $C$ denote the centre of $\s$.
Then $C$ is snc with respect to $E$
because, in the notation above, $C$ lies inside every element of $E^1(a) \cup \cdots E^k(a)$ and $C$ is snc with respect to $\cE^k(a)$. Since $C \subset S$, it follows
that $\s$ is $\inv_k$-admissible. By Lemma \ref{lem:snc}, $C$
contains no snc points (since some $\mu_{H,k+1}(a) \neq 0$, for all $a \in C$). 

Since $d_k=1$, $C$ is of the form $N^k \cap H$, for a single $H \in \cE^k(a)$;
i.e., $C$ is of codimension $1$ in $N^k$. Therefore, $\s$ induces an isomorphism
$(N^k)' \to N^k$, where $(N^k)'$ denotes the strict transform of $N^k$.
\end{remarks}  

\begin{lemma}\label{lem:clean}
Assume that $\inv \leq \inv_{p,s}$ on $X = X_m$, in some year $m$ of the desingularization history. Consider the cleaning sequence for
$S = (\inv_k = (\inv_{p,s})_k)$ (Definition \ref{def:clean}). Then, over the cleaning
sequence, we can define maximal contact subspaces of every codimension involved, 
as well as (a modification of) $\inv$ which remains both semicontinuous and infinitesimally semicontinuous.
\end{lemma}

\begin{proof}
Consider the first blowing up $\s$ in the cleaning sequence, at a point $a$
as above. Let $C$ denote the centre of $\s$.

Since $d^k=1$, $C\subset N^k$ is given by $H\cap N^k$, for some 
$H\in\mathcal{E}^k(a)$. Let $(\underline{\mathcal{I}}^k)' = 
(Z',(N^k)',(\mathcal{E}^k)',(\mathcal{I}^k)',1)$ denote the transform of 
$\ucI^k$ by $\s$ (in particular, $Z = Z_{m+1}$; i.e., this is year $m+1$).
Then $\s|_{(N^k)'}: (N^k)' \to N^k$ is an isomorphism (which
we consider to be the identity), so that $\s^*(\cI^k) = \cI^k$. Therefore
$(\cI^k)' = \cI_H^{-1}\s^*(\cI^k) = \cI_H^{-1}\cI^k$ and
$$
\cR((\cI^k)') = \cR(\cI^k), \quad \cM((\cI^k)') = \cI_H^{-1}\cM(\cI^k).
$$

Over the complement of $C$, $\s$ is an isomorphism and $\inv$ will be
unchanged. Consider $a'\in \sigma^{-1}(a)$. 
If $a' \notin (N_k)'$, then $\inv_k(a')<\inv_k(a)$, by the definition of maximal contact
(since $\s$ is $\inv_k$-admissible). For any $a'\in \sigma^{-1}(a)$ such that
$\inv_k(a')<\inv_k(a)$,
$\inv_k(a')$ can be extended to $\inv(a')$ as in the usual desingularization algorithm. 
(This is ``year zero" for $\inv_k(a')$.)

On the other hand, suppose $a\in (N_k)'$ and $\inv_k(a')=\inv_k(a)$. Following the standard desingularization
algorithm, we set $\nu_{k+1}(a')= \ord_{a'}\cR((\cI^k)') = \nu_{k+1}(a)$. Then
$\inv_{k+1/2} = (\inv_k,  \nu_{k+1})$ is semicontinuous on $Z' = Z_{m+1}$ and
infinitesimally semicontinuous through to year $m+1$. The exponents 
$\mu_{H,k+1}(a')$ in $\cM((\cI^k)')$, and $\nu_{k+1}(a')$ are invariants of the
equivalence class of $(\ucI^k)'$ (cf. \cite[Def.\,5.10 ff.]{BMmin}).

Recall that the companion ideal $\ucJ^k$ of $\ucI^k$ is $(\cR(\cI^k),1)$ (as in the proof
of Lemma \ref{lem:norm}). The equivalence class of the companion ideal
$\ucJ_{m+1}^k = (\ucJ^k)'$ of $\ucI_{m+1}^k = (\ucI^k)'$ depends 
only on the equivalence
class of $(\ucI^k)'$ (cf. \cite[Cor.\,5.3]{BMfunct}).

If $\nu_{k+1}(a)>0$, then $\cJ^k$ admits a maximal contact hypersurface 
$N^{k+1} \subset N^k$ at $a$ (relative to $\cE^{k+1}(a)$; in particular, $N^{k+1}$ is transverse to $\cE^{k+1}$ at $a$). 
If we were to follow the standard desingularization algorithm,
then we would take $s_{k+1}(a'): = \#E^{k+1}(a')$, where $E^{k+1}(a')$ is the set of
strict transforms at $a'$ of elements of $E^{k+1}(a)$, and we would include the
exceptional divisor $\tH = \s^{-1}(C)$ of $\s$ in $\cE^{k+1}(a')$. There are two cases to be considered. 
\smallskip

\noindent
(1) $H\notin E^{k+1}(a)$. Then $s_{k+1}(a')=s_{k+1}(a)$. Since $H \in \cE^{k+1}(a)$, 
$H$ is transverse to $N^{k+1}$. Therefore, $\tH$ is transverse to $(N^{k+1})'$,
and $(N^{k+1})'$ is a valid maximal contact subspace of codimension $k+1$ at $a'$.
\smallskip

\noindent
(2) $H\in E^{k+1}(a)$. Then $H$ is not necessarily transverse to $N^{k+1}$ and
$\tH := \s^{-1}(C)$ is not necessarily transverse to 
$(N^{k+1})'$, so that $(N^{k+1})'$ is not valid as a maximal contact subspace of
codimension $k+1$ at $a'$. Instead, we modify the definition of $\inv_{k+1}$
and the algorithm in this case by continuing
to count $\tH$ as an element of $E^{k+1}(a')$ (so we are viewing $(N^k)' \cap \tH$
simply as a relabelling of $N^k \cap H$; i.e., $H$ is moved away from $N^k$ and
replaced by $\tH$). Thus $s_{k+1}(a') = s_{k+1}(a)$ and $\cE^{k+1}(a')$ is still
transverse to $(N^{k+1})'$.
\smallskip

In both cases, it follows that the equivalence class
of the coefficient ideal plus boundary $\ucI_{m+1}^{k+1}$ associated to 
$\ucJ_{m+1}^k = (\ucJ^k)'$
depends only on the equivalence class of $\ucJ_{m+1}^k$ and $E^{k+1}(a')$
\cite[Sect.\,5, Step I, Case B]{BMfunct}. Therefore, $\inv_{k+1} = (\inv_k, \nu_{k+1},
s_{k+1})$ extends as in the standard desingularization algorithm to $\inv$ on $Z_{m+1}$.
Clearly, if $\inv_k(a') = \inv_k(a)$, where $a \in C$ as above, then $\inv(a') = \inv(a)$.

The argument above can be iterated over the full sequence of cleaning blowings-up.
The order of the blowings-up in the cleaning sequence (i.e., in desingularization of
the corresponding monomial marked ideal) is determined by the ordering of
$\cE^k$ given by the year of birth of each of its elements see \cite[Sect.\,5, Step II,
Case A]{BMfunct}.
\end{proof}

\begin{remark}\label{rem:why}
After cleaning the loci $(\inv_k = (\inv_{p,s})_k)$, for all $k$, $(\inv = \inv_{p,s})$
becomes snc. We will then continue to blow up with closed centres which lie in
the complement of $\{\text{snc}\}$ (Section \ref{sec:alg}). The purpose of defining
$\inv$ over the cleaning sequences is to ensure that, in the complement
of $\{\text{snc}\}$, we will only have to consider values $\inv_{p',s'} < \inv_{p,s}$
in order to resolve all but $\{\text{snc}\}$ after finitely many steps. If,
after cleaning $(\inv = \inv_{p,s})$, we were to apply the resolution algorithm in the
complement of $\{\text{snc}\}$, beginning as if in year zero, we might introduce
points where $\inv = \inv_{p',s'} > \inv_{p,s}$.
\end{remark}


\section{Algorithm for the main theorem}\label{sec:alg}
In this section we prove Theorem \ref{thm:main}. Let $\cS$ denote $\{\inv_{p,s}\}$ where $s=(s_1,...,s_d)$ and $p,d,|s|\leq n:= \dim Z$. Then $\cS$ is finite and totally ordered. 
Consider the following two steps.
\smallskip

\noindent
(1) Let $\inv_{p,s}$ be the maximum element of $\cS$. Follow the desingularization algorithm of \cite{BMinv}, \cite{BMfunct} to decrease
$\inv$ until $\inv$ is everywhere $\leq \inv_{p,s}$. Then blow-up any component of the locus $(\inv=\inv_{p,s})$ that contains only non-snc points. The result is that
$(X,E)$ is generically snc on every component of the locus $(\inv=\inv_{p,s})$. 
By Remark \ref{rem:sigma}, $(\inv=\inv_{p,s})\subset \Sigma_p(X)$.
\smallskip

\noindent
(2) Clean the locus 
$(\inv_k=(\inv_{p,s})_k)$, successively for each $k=r-1,r-2\ldots,1$, using Lemma
\ref{lem:clean}. Then all $\mu_{H,k}=0$, $2 \leq k \leq r$, so, by 
Theorem \ref{thm:char}, all points of $(\inv=\inv_{p,s})$ are snc. Therefore, $(X,E)$ is snc on $Y_{p,s}:= (\inv\geq \inv_{p,s})$, and hence in a neighbourhood of $Y_{p,s}$. 
\smallskip

We can therefore repeat the above two steps on $X_{p,s}:=X\setminus Y_{p,s}$,
replacing $\cS$ by $\cS\setminus\{\inv_{p,s}\}$. The centres of all blowings-up involved are closed in $X$ since they are closed in $X_{p,s}$ and contain no snc points
(according to Theorem \ref{thm:char} and Remarks \ref{rem:clean}).

Since $\cS$ is finite, the process terminates after finitely many iterations, so that $(X,E)$ 
becomes everywhere snc.

The desingularization algorithm \cite{BMfunct} is functorial with respect to 
\'etale or smooth morphisms. The snc condition is preserved
by \'etale or smooth morphisms which preserve the number of irreducible components at every point. The functoriality statement in Theorem \ref{thm:main} is an immediate
consequence (cf. \cite[Sect.\,9]{BV}). This completes the proof. \hfill \qedsymbol

\begin{remark}\label{rem:final}
The algorithm above admits certain variations; e.g., (1) can be changed so 
that, when $\inv \leq \inv_{p,s}$, we blow up only components of 
$(\inv=\inv_{p,s})$ that are disjoint from $\Sigma_p(X)$ (Remark \ref{rem:sigma}
still applies).
\end{remark}

\begin{example}\label{ex2}
Consider $X = (x(x+yz)=0)$, $E = \emptyset$. The standard resolution algorithm A first blows up $C_0 = \{0\}$ and afterwards gives:

\renewcommand{\arraystretch}{1.5}
\begin{longtable}{c | c | c | c | c}
\hline
year $j$ & chart & marked ideal $\ucI^j$ & $\inv(0)$ & centre $C_j$\\\hline\hline
1 & $(xy,y,yz)$ & $(xy(x+yz), 1)$ & $(2,0,1,1,1,0,\infty)$ & $\{0\}$\\\hline
2 & $(xy,y,yz)$ & $(xy^2(x+yz), 1)$ & $(2,0,1,0,\infty)$ & $\{x=z=0\}$ \\\hline
\end{longtable}

\noindent
In year 1, $\inv(0)=\inv_{2,s}$, $s=(0,1)$, but Algorithm C (above)
also blows up $C_1 = \{0\}$.
In year 2, $\inv(0)=\inv_{2,0}$; Algorithm C (and B) continues with a cleaning blow-up, center $(x=y=0)$. The variant of C in Remark \ref{rem:final}, 
however, cleans up with centre $(x=y=0)$ 
in year 1. After either cleaning blow-up, we have snc over the chart shown.
\end{example}

\bibliographystyle{alpha}

\end{document}